\documentclass[12pt]{article}
\usepackage{amsfonts}
\usepackage{fancyhdr}

\usepackage{graphicx}
\usepackage{amsmath}

\newtheorem{theorem}{Theorem}

\newtheorem{corollary}[theorem]{Corollary}

\newtheorem{definition}[theorem]{Definition}

\newtheorem{lemma}[theorem]{Lemma}

\newtheorem{remark}[theorem]{Remark}

\newenvironment{proof}[1][Proof]{\textbf{#1.} }{\ \rule{0.5em}{0.5em}}

\begin{document}

\title{{\LARGE Random eigenvalues from a stochastic heat equation.
\footnote{This research was partially supported by the CONACYT.}}}
\author{  
Carlos G. Pacheco
\thanks{Departamento de Matematicas, CINVESTAV-IPN, A. Postal 14-740, Mexico D.F. 07000, MEXICO. Email: cpacheco@math.cinvestav.mx}}

\maketitle

\begin{abstract}
In this paper we prove the convergence of the eigenvalues of a random matrix that approximates a random Schr\"{o}dinger operator.
Originally, such random operator arises from a stochastic heat equation.  
The proof uses a detailed topological analysis of certain spaces of functions where the operators act.
\end{abstract}

{\bf 2000 Mathematics Subject Classification: 15B52, 47H40}
\\

\textbf{Keywords:} Stochastic heat equation, weak stochastic operator, random matrix, spectrum, eigenvalues.


\section{Stochastic heat model}

Stochatic partial differential equations (SPDE) has grown significantly in applied and pure mathematics.
In particular, the stochastic heat equation is consider a natural model for different phenomena, see e.g. \cite{Walsh}. 
For these reasons it is natural to consider discrete approximation of these models, say, to have a numerical procedure to solve it.
In Pacheco \cite{Pacheco}, it was proposed at random matrix to study a stochastic heat equation (SHE), or more precisely, to approximate the one-dimensional random operator associated to the SHE. 
In that paper it was proved weak convergence using the inner product, which was done by identifying the matrix with a composition using a projection.

In the current paper we prove the convergence of the spectrum, which in this case means the convergence of the eigenvalues.
It is our objective to show how the eigenvalues of the random matrices converges to the spectrum of the stochastic operator.
To do that, we use the variational formulae for eigenvalues of selfadjoint operators to connect with the min-max representation of eigenvalues in the Courant-Fisher theorem.

Let us talk about the stochastic equation and its operator. 
The SPDE that we have in mind is the following,
\begin{eqnarray}
\label{SBVP}
& & \frac{\partial u}{\partial t}=\beta \frac{\partial^{2}u}{\partial x^{2}}+uw^{\prime},\ t>0,\ x\in[0,1],
\end{eqnarray}
where $w^{\prime}$ represents Gaussian space-time noise. 

Then, we could concentrate in the following associated one-dimensional operator, 
\begin{equation}
\label{SchL}
L u:=\beta \frac{d^{2}u}{dx^{2}}+u\times b^{\prime},\ x\in [0,1], 
\end{equation}
where $b^{\prime}$ is a Gaussian white noise on the interval $[0,1]$. 
Operator $L$ is consider to be a random Schr\"{o}dinger operator and one can properly define it using inner products, this is done in Definition \ref{DefL}.

The proposed random matrix in \cite{Pacheco} to approximate $L$ is
$A_{n}:=$
\begin{equation}
\label{Tmatrix1}
\left[
\begin{array}{cccc}
\sqrt{n+1}\xi_{1} -2 \beta (n+1)^{2}   & \beta (n+1)^{2} &             &            \\
\beta (n+1)^{2} &  \sqrt{n+1} \xi_{2} -2 \beta (n+1)^{2}  &   \beta (n+1)^{2}    &    \\
&  &  \ddots &  \\
      &                             & \beta (n+1)^{2} & \sqrt{n+1}\xi_{n} -2 \beta (n+1)^{2}
\end{array}
\right],
\end{equation}
where $\xi_{1},\ldots, \xi_{n}$ are i.i.d. $N(0,1)$ r.v.s.

Loosely speaking, consider the operator $L_{n}= A_{n}P_{n}$ which is the composition of a projection and the random matrix.
The main result in \cite{Pacheco} is the following convergence, as $n\to \infty$, 
\begin{equation*}
\langle L_{n}u,v\rangle\to \langle Lu,v\rangle
\end{equation*}
in mean square for every pair of functions $u$ and $v$.
It turns out that the convergence just described does not imply the convergence of the spectrum.

Here, we are interested in proving convergence of the eigenvalues to the spectrum of $L$. 
In this study, it is not used the composition $L_{n}$, instead we calculate the eigenvalues and check that they approximate the spectrum of $L$. 

We would like to mention that this work was motivated by the one in \cite{Ramirez}, where it is also study the convergence of the eigenvalues of a random matrix to the spectrum of a random operator. 

\subsection{One-dimensional operators}\label{Ss1DimOp}

In this section we properly define the random operator we deal with, this is done following ideas taken from \cite{Skorohod}.

Now, in order to define $L$ in a rigorous way, we first set the space 
\begin{equation}
H_{1}:=\left\{ h\in L^{2}[0,1]: h\text{ absolutely continuous},\ h^{\prime}\in L^{2}[0,1] , \ h(0)=h(1)=0 \right\},
\end{equation}
which is dense (see for instance Example 1.11 of Chapter X in \cite{Conway}) in the Hilbert $H:=L^{2}[0,1]$. 
It is also known (see e.g. \cite{Schmudgen}) that $H_{1}$, with the norm $\|h\|:=\|h\|_{2}+\|h^{\prime}\|_{2}$, is a Sobolev space, which is in fact a separable Hilbert space, and as such it has a countable orthonormal base; we will refer to this base when proving Theorem \ref{TheoremMain}, specifically in Lemma \ref{lemaM}.   

Using integration by parts, we can define $L$ by defining $\langle Lu,v\rangle$, for every $u,v\in H_{1}$. 
Here $\langle \bullet,\bullet \rangle$ stands for the inner product in $L_{2}$ and we will also write $\|\bullet\|$ for the norm in $L_{2}$.

\begin{definition}\label{DefL}
The operator $L$ associated to the expression
\begin{equation}
L u:=\beta \frac{d^{2}u}{dx^{2}}+u\times b^{\prime},\ x\in [0,1], 
\end{equation}
with $b^{\prime}$ being the white noise on $[0,1]$, is defined weakly in the following way. For every $u,v\in H_{1}$
\begin{equation}
\langle Lu,v\rangle:=-\beta\int_{0}^{1}u^{\prime}(x)v^{\prime}(x)dx+\int_{0}^{1}u(x)v(x)dB(x),
\end{equation} 
where $B$ is a Brownian motion on $[0,1]$.
\end{definition}

Another useful way to write $L$, using It\^{o}'s formula, is 
$$\int_{0}^{1}u(x)v(x)dB(x)=-\int_{0}^{1}(u^{\prime}(x)v(x)+u(x)v^{\prime}(x))B(x)dx.$$
In fact, this expression was originally used in \cite{Fuku} to analyze the spectrum.
We extract the following result from \cite{Fuku}.
\begin{theorem} \label{FuNaThm} \textbf{(Fukushima and Nakao (1977)).} Consider the one-dimensional random Schr\"{o}dinger operator
\begin{equation}\label{L0}
L_{0}:= -\frac{d^{2}}{dx^{2}} + b^{\prime},
\end{equation}
defined weakly as follows. For every $u,v \in H_{1}$,
\begin{equation*}
\langle L_{0} u,v\rangle:= \int_{0}^{1}u^{\prime}(x)v^{\prime}(x)dx-\int_{0}^{1}\{u^{\prime}(x)v(x)+u(x)v^{\prime}(x)\}B(x)dx.
\end{equation*}
Then $L_{0}$ has a discrete spectrum $\{\lambda_{1},\lambda_{2},\ldots\}$ and it can be calculated as
\begin{equation}
\label{eigenHk}
\lambda_{k}=\inf_{M_{1}\subset H_{1}\atop dim(M_{1})=k}\sup_{v\in M_{1}\atop \|v\|=1}\langle L_{0} v, v\rangle.
\end{equation}
\end{theorem}

We can adapt previous result to obtain the
\begin{corollary} 
$L$ has a discrete spectrum $\{\lambda_{1},\lambda_{2},\ldots\}$, which can be obtained through
\begin{equation}
\label{eigenvk}
\lambda_{k}\overset{(d)}{=}-  \inf_{M_{1}\subset H_{1}\atop dim(M_{1})=k}\sup_{v\in M_{1}\atop \|v\|=1}
\left\{ \beta \int_{0}^{1}(v^{\prime}(x))^{2}dx+\int_{0}^{1}v^{2}(x)dB(x)\right\}.
\end{equation}
\end{corollary}
\begin{proof}
Notice that
\begin{eqnarray*}
-\frac{1}{\beta}Lf&=&-\frac{d^{2}}{dx^{2}}f-fw^{\prime}\\
&\overset{(d)}{=}& -\frac{d^{2}}{dx^{2}}f+fw^{\prime}.
\end{eqnarray*}
That is, for all $u,v\in H_{1}$, 
$$-\frac{1}{\beta}\langle Lu,v\rangle\overset{(d)}{=}\langle H_{\beta}u,v\rangle,$$ 
where $H_{\beta}$ is defined as $H$ but with a Brownian motion $B_{\beta}$ with variance $\beta^{-2}$. 
Then, the eigenvalues of $L$ can be calculated as those of $H_{\beta}$. Theorem \ref{FuNaThm} can be stated for $H_{\beta}$, and the eigenvalues of $H_{\beta}$ become  
\begin{equation*}
\eta_{k}=\inf_{M_{1}\subset H_{1}\atop dim(M_{1})=k}\sup_{v\in M_{1}\atop \|v\|=1}
\int_{0}^{1}(v^{\prime}(x))^{2}dx+\int_{0}^{1}v^{2}(x)dB_{\beta}(x).
\end{equation*}
Which defines in distributions the eigenvalues of $L$ as $\lambda_{k}\overset{(d)}{=}-\beta \eta_{k}$, i.e. formula (\ref{eigenvk}).
\end{proof}


\subsection{Weak convergence}\label{SubSWC}

Originally, the matrix $A_{n}$ came from the matrix
\begin{equation*}
A_{n}:=\left[
\begin{array}{cccc}
\frac{X_{1}^{(n)}}{\Delta x}-2 \beta \frac{1}{(\Delta x)^{2}}   & \beta \frac{1}{(\Delta x)^{2}} &             &            \\
\beta \frac{1}{(\Delta x)^{2}} & \frac{X_{2}^{(n)}}{\Delta x}-2 \beta \frac{1}{(\Delta x)^{2}}  &   \beta \frac{1}{(\Delta x)^{2}}    &    \\
&  &  \ddots &  \\
      &                             & \beta \frac{1}{(\Delta x)^{2}} & \frac{X_{n}^{(n)}}{\Delta x}-2 \beta \frac{1}{(\Delta x)^{2}}
\end{array}
\right].
\end{equation*}
Here, $\Delta x= 1/(n+1)$ is the size partition of the space, and $X_{i}^{(n)}$ are independent normal random variables with mean $0$ and variance $\Delta x$.
Moreover, if $u:[0,1]\to \mathbb{R}$ is some well behaved function, we can construct the vector $v:=(u(x_{1}),\ldots, u(x_{n}))^{T}$, where $\{x_{j}, j=1,\ldots,n\}$ represents the partition of the interval $[0,1]$ in subintervals of size $\Delta x$; $T$ means the traspose.
Then, the $i$ entry of the multiplication $A_{n}v$ is given by 
\begin{equation}
\label{An}
[A_{n}u]_{i}:=
\beta \frac{u(x_{i+1})-2u(x_{i})+u(x_{i-1})}{(\Delta x)^{2}} +u(x_{i}) \frac{X_{i}^{(n)}}{\Delta x}.
\end{equation}

Let us give precisely the specifications of the partitions.
Let $\Delta^{(n)}=1/(n+1)$ and $\Pi_{n}:=\{ x_{0}^{(n)}, x_{1}^{(n)},\ldots , x_{n+1}^{(n)}\}, \ n=1,2,\ldots$ be a sequence of partitions such that $|x^{(n)}_{k+1}-x^{(n)}_{k}|=1/(n+1)$, with $x_{0}^{(n)}=0$ and $x_{n+1}^{(n)}=1$. 
Thus $u(x_{0}^{(n)})=u(x_{n+1}^{(n)})=0$ for every $n\geq 1$ and $u\in H_{1}$. 
Now, let $H_{1}^{(n)}\subset H_{1}$ be the linear subspace of stepwise functions $v$ which are constant on the intervals $[x_{i}^{(n)},x_{i+1}^{(n)})$, $i=0,1,\ldots, n$, 
and let 
$$P_{n}:H_{1}\to H_{1}^{(n)}$$ 
be the projection that associates to any function $u\in H_{1}$ a stepwise right continuous function $u_{n}$ that takes the values $u(x_{k}^{(n)})$, i.e. 
\begin{equation*}
u_{n}(x):=\left\{
\begin{array}{cl}
0 & x\in [0,x_{1}^{(n)})\\
u(x_{i}^{(n)}) & x\in [x_{i}^{(n)},x_{i+1}^{(n)}), i=1,\ldots, n\\
0 &  x=x_{n+1}^{(n)}=1.
\end{array}
\right.
\end{equation*}
We now consider the operator $L_{n}:=A_{n} P_{n}$ whose image in $H_{1}^{(n)}$ are stepwise function coming from multiplying the matrix $A_{n}$ to the $n$-vector associated to $u_{n}:=P_{n}u$, with $u\in H_{1}$ at points 
$x_{1}^{(n)},\ldots, x_{n}^{(n)}$. 
That is, if $v_{n}:=(u_{n}(x_{1}^{(n)}),\ldots, u_{n}(x_{n}^{(n)}))^{T}$, then $L_{n}u$ is the step functions whose constant values are $A_{n}v_{n}$.

What it is now known is the
\begin{theorem}
\label{ThWC} (Pacheco \cite{Pacheco})
For every pair $u,v\in H_{1}$, $\langle L_{n}u,v\rangle\overset{n\to \infty}{\to} \langle Lu,v\rangle$ in mean square.
\end{theorem}


\subsection{Convergence of eigenvalues}

In Theorem \ref{ThWC}, using a projection and the matrix $A_{n}$ it is consider a new operator $L_{n}$ to prove weak convergence. 
Now, we simply want to calculate the eigenvalues of the matrix and see if they converge somehow to the spectrum of $L$.

\begin{theorem}
\label{TheoremMain}
The $k-$th eigenvalue $\lambda_{k}^{(n)}$ of $A_{n}$ converges in distribution as $n\to \infty$ to the $k-$th eigenvalue of $L$.
\end{theorem}
To obtain the result, without loss of generality let us take $\beta=1$, hence it will suffice to prove the result for $-A_{n}$ and $L_{0}$ in (\ref{L0}).
To facilitate notation we will omit superscript $(n)$ in the partition and size-partition.



\begin{lemma}\label{LemaFn} The eigenvalue $\lambda_{k}^{(n)}$ admits the following representation,
\begin{equation}\label{EqFn}
\lambda_{k}^{(n)}\overset{(d)}{=}
\inf_{M_{1}\subset H_{1}\atop dim(M_{1})=k}\sup_{g\in M_{1}\atop 1=\sum_{i=1}^{n} g(x_{i})^{2}\Delta x} F_{n}(g),
\end{equation}
where 
\begin{eqnarray}\label{Fn}
\nonumber F_{n}(g)&:=& 
\sum_{i=1}^{n} \left( \frac{ g(x_{i+1})- g(x_{i}) }{\Delta x}\right)^{2}\Delta x + \sum_{i=1}^{n} g(x_{i})^{2}X_{i} \\ 
& + & \frac{g(x_{1})-g(x_{0})}{\Delta x}g(x_{1})+ \frac{g(x_{n+1})-g(x_{n})}{\Delta x}g(x_{n}),
\end{eqnarray}
where $\{x_{0},\ldots,x_{n+1}\}=\Pi_{n+1}$.
\end{lemma}
\begin{proof}
According to Courant-Fisher Theorem (see e.g. \cite{Zhang}, pp. 268), setting $u(0)=u(n+1)=0$, the eigenvalues of $-A_{n}$ can be calculated as
\begin{equation}\label{lambdan}
\lambda_{k}^{(n)}= \min_{M_{0}\subset \mathbb{R}^{n} \atop dim(M_{0})=k}\max_{u\in M_{0}\atop \|u\|=1}
\sum_{i=1}^{n}\left\{ \frac{-1}{(\Delta x)^{2}}(u(i+1)-2u(i)+u(i-1))-u(i)\frac{X_{i}}{\Delta x} \right\}u(i),
\end{equation}
for $n\geq k$. 

Notice that $1=\sum_{i=1}^{n}u(i)^{2}=\sum_{i=1}^{n}\tilde{u}^{2}(i)\Delta x$ with $u(i)=\tilde{u}(i)\sqrt{\Delta x}$. We can then substitute in (\ref{lambdan}), use that 
$-X_{i}\overset{(d)}{=}X_{i}$ and rewrite as done in \cite[Theorem 6]{Pacheco} to obtain that
\begin{eqnarray}
\lambda_{k}^{(n)}&\overset{(d)}{=}& \min_{M_{0}\subset \mathbb{R}^{n}\atop dim(M_{0})=k}\max_{\tilde{u}\in M_{0}\atop 1=\sum\tilde{u}(i)^{2}\Delta x}
\Bigg[\sum_{i=1}^{n} \left( \frac{ \tilde{u}(i+1)- \tilde{u}(i) }{\Delta x}\right)^{2}\Delta x + \sum_{i=1}^{n} \tilde{u}(i)^{2}X_{i} \nonumber\\
 &+& \frac{1}{\Delta x}\{ (\tilde{u}(1)-\tilde{u}(0))\tilde{u}(1)+ (\tilde{u}(n+1)-\tilde{u}(n))\tilde{u}(n) \} \Bigg].\nonumber
\end{eqnarray}

Since we always have $\tilde{u}(0)=\tilde{u}(n+1)=0$, there is $g\in H_{1}$ which coincides with the $n+2$ vector $\tilde{u}$ when evaluating at an equidistant partition $\Pi_{n+1}$. 
Conversely, for every $g\in H_{1}$ there exists a  $n+2$ vector whose entries are the values of $g$ at $\Pi_{n+1}$.  
Then we can change $\min\{M_{0}\subset \mathbb{R}^{n}, dim(M_{0})=k\}$ by $\inf\{M_{1}\subset H_{1}, dim(M_{1})=k\}$ and write down expression (\ref{EqFn}).
\end{proof}

\begin{remark}\label{RemMeaning}
We can give the following interpretation of $F_{n}(g)$ of previous result. 
For the first sum, project first $g$ into a piecewise linear function by joining the points $g(x_{0}),\ldots, g(x_{n+1})$ with straight lines 
(call this projection $\hat{g}$), and then calculate 
$$\int (\hat{g}^{\prime}(x))^{2}dx.$$ 
For the second sum, project first $g$ into a stepwise function with constant values given by $g(x_{0}),\ldots, g(x_{n+1})$ 
(call this projection $\bar{g}$), and then, taking into account Section \ref{SubSWC}, calculate 
$$\int (\bar{g}(x))^{2}dB(x).$$ 
For the last two terms we do not worry because they will vanish as the partition becomes finer; 
this is so because the quotients will converge to the derivatives and the evaluations of $g$ to zero.
\end{remark}

Our aim is to prove that when $n\to\infty$, $\lambda_{k}^{(n)}$ converges in distribution to (\ref{eigenHk}), which is
\begin{equation*}
\lambda_{k}=\inf_{M_{1}\subset H_{1}\atop dim(M_{1})=k}\sup_{g\in M_{1}\atop \|g\|=1} \int_{0}^{1}(g^{\prime}(x))^{2}dx+\int_{0}^{1}g^{2}(x)dB(x),
\end{equation*}
written shortly as
\begin{equation}\label{EqSupDisc}
\lambda_{k}=\inf_{M_{1}\subset H_{1}\atop dim(M_{1})=k} \sup_{g\in M_{1}\atop 1=\sum g(x_{i})^{2}\Delta x} F(g),
\end{equation}
where 
\begin{equation}\label{FormulaF}
F(g):=\int_{0}^{1}(g^{\prime}(x))^{2}dx+\int_{0}^{1}g^{2}(x)dB(x).
\end{equation}
Now we proceed to realize what really is the space where one is taking supremum. 

For any subspace $M_{1}\subset H_{1}$ of dimension $k$, it is known that the set 
$$S:=\{ g\in M_{1} : \|g\|=1 \} $$ 
is homeomorphic to the $k$ dimensional sphere 
$$S^{k-1}:=\left\{\alpha\in\mathbb{R}^{k}: \sum_{i=1}^{k}\alpha_{i}^{2}=1\right\},$$
which is denoted by the map $\gamma: S^{k-1}\to S$. 
We will check that the set 
$$
S_{n}:=\{ g\in M_{1} : \sum_{i=1}^{n} g(x_{i})^{2}\Delta x=1 \}
$$
is also homeomorphic to $S^{k-1}$.
The following result will tell us how $S_{n}$ becomes $S$ as the partition becomes finer, and ultimately that $F_{n}$ in (\ref{Fn}) converges pointwise to $F$ in (\ref{FormulaF}).

\begin{lemma}\label{LemaSup} 
Consider $M_{1}\subset H_{1}$ with $dim(M_{1})=k$.
Then, there is an homomeorphism $\gamma^{(n)}: S^{k-1}\to S_{n}$.
Furthermore, let $g:=\gamma(\alpha)$ and $g_{n}:=\gamma^{(n)}(\alpha),\ n=1,2,\ldots$ for each $\alpha\in S^{k-1}$. 
Then
\begin{equation*} 
F_{n} (g_{n})\to F( g ) \text{ in mean square as }  n\to\infty.
\end{equation*}
\end{lemma}
\begin{proof}
Let $E:=\{e_{1},\ldots,e_{k}\}$ be a orthonormal base of $M_{1}$.
Let $g_{n}\in M_{1}$ with $\sum_{i=1}^{n}g_{n}^{2}(x_{i})\Delta x=1$. 
Since $E$ is a base, there is $\alpha_{n}:=(\alpha_{1}^{(n)},\ldots,\alpha_{k}^{(n)})^{T}$ such that $g_{n}=\sum_{i=1}^{k}\alpha_{i}^{(n)} e_{i}$. 
Thus
\begin{eqnarray*}
1&=& \sum_{i=1}^{n}\left(\sum_{j=1}^{k}\alpha_{j}^{(n)}e_{j}(x_{i})\right)^{2}\Delta x= \alpha^{T}_{n}U_{n}U_{n}^{T}\alpha_{n},
\end{eqnarray*}
where 
\begin{equation*}
\label{Tmatrix}
U_{n}:=\sqrt{\Delta x}\left[
\begin{array}{ccccc}
e_{1}(x_{1})  &  &   \ldots          &          &  e_{1}(x_{n}) \\
\vdots &  &   & & \\
e_{k}(x_{1})      &           &    \ldots              &  &  e_{k}(x_{n})
\end{array}
\right].
\end{equation*}
Define $f_{n}(\alpha):= \alpha^{T}U_{n}U_{n}^{T}\alpha$, which is a $\mathbb{R}^{k}\to\mathbb{R}$ continuous function. 
Since $f_{n}$ has $k$ parameters, the set $\tilde{S}_{n}:=f_{n}^{-1}(1)\subset \mathbb{R}^{k}$ is homeomorphic to $S_{n}$and to $S^{k-1}$.
These facts help to see that there is an homomeorphism $\gamma^{(n)}:S_{n}\to S^{k-1}$.

Now, for $1\leq i,j\leq k$ notice that the $(i,j)$-entry of the matrix $U_{n}U_{n}^{T}$ is of the form $\sum_{r=1}^{n} e_{i}(x_{r})e_{j}(x_{r})\Delta x$.
So
\begin{equation*}
\lim_{n\to\infty} U_{n}U_{n}^{T} (i,j)=
\left\{
\begin{array}{cc}
0 & i\neq j\\
1 & i=j.
\end{array}
\right.
\end{equation*}
The convergence is uniform in $i$ and $j$ because the number of pairs $(i,j)$ is finite.
Then
\begin{equation}\label{Uconv}
\alpha^{T}U_{n}U_{n}^{T} \alpha  \to \|\alpha\| \text{ as } n\to\infty,
\end{equation}
for each $\alpha \in \mathbb{R}^{k}$. 
Therefore $\tilde{S}_{n}$ is deformed into $ S^{k-1}$ when $n\to \infty$. 
More precisely, let $\alpha\in S^{k-1}$ and take the associated 
$\rho_{n}:= (\rho_{1}^{(n)}, \ldots, \rho_{k}^{(n)} )^{T}\in \tilde{S}_{n}$. 
Then $\rho_{n}\to \alpha$ as $n\to \infty$ in the Euclidean norm.

From the above convergence, let us check that $g_{n}\overset{L_{2}}{\to}g$. 
Indeed, this is so because
\begin{eqnarray*}
\int_{0}^{1}(g_{n}(x)-g(x))^{2}dx
&=&\int_{0}^{1}\left( \sum_{i=1}^{k}\rho_{i}^{(n)}e_{i}(x)-\sum_{i=1}^{k}\alpha_{i}e_{i}(x) \right)^{2}dx\\
&\leq& k \sum_{i=1}^{k} \left( \rho_{i}^{(n)} -\alpha_{i}\right)^{2}\int_{0}^{1}e_{i}^{2}(x)dx
=k \sum_{i=1}^{k}\left( \rho_{i}^{(n)}-\alpha_{i}\right)^{2}.
\end{eqnarray*}
In a similar way we can prove that $g_{n}^{\prime}\overset{L_{2}}{\to}g^{\prime}$. 
Furthermore, recalling Remark \ref{RemMeaning}, one can check that
$\hat{g}_{n}\overset{L_{2}}{\to}g$ and $\bar{g}_{n}\overset{L_{2}}{\to}g$, which also helps to show that
\begin{equation*}
\int (\hat{g}_{n}^{\prime}(x))^{2}dx\to \int (g^{\prime}(x))^{2}dx, 
\text{ and that }
\int (\hat{g}_{n}(x))^{2}dB(x)\overset{L_{2}}{\to} \int (g(x))^{2}dB(x).
\end{equation*}
Therefore for each $\alpha\in S^{k-1}$, $F_{n}(g_{n})\overset{L_{2}}{\to}F(g)$ as $n\to\infty$;
\end{proof}


Now we need the following ingredient. 
Let us now identify the set where the infimum is being taken in the representations (\ref{EqFn}) and (\ref{EqSupDisc}) of the eigenvalues. 
That is, in the set
$$\mathcal{M}:=\{M_{1}\subset H_{1}: dim(M_{1})=k\}.$$ 

Let $B$ be the unit sphere in $l_{2}$. 
It turns out that $\mathcal{M}$ has naturally a topology inherited from the product-topology of 
\begin{equation}\label{SetBk}
B_{k}:=\underbrace{B\times \ldots \times B}_{k\ \text{ times }}.
\end{equation}
More specifically:
\begin{lemma}\label{lemaM}
The set $\mathcal{M}$ has a relatively compact topology with a countable dense subset.
\end{lemma}  
\begin{proof}
Let $V:=\{ h_{1}, h_{2},\ldots\}$ be a countable orthonormal base of $H_{1}$.
Take any $M_{1}\in \mathcal{M}$, which is generated by $k$ elements of $H_{1}$ of norm 1, say $E:=\{e_{1},\ldots,e_{k}\}$. 
Each $e\in E$ is a linear combination of $V$, that is $e=\sum_{i=1}^{\infty}\alpha_{i}h_{i}$, where 
$\alpha=(\alpha_{1},\alpha_{2},\dots)$ is an element of the set of square summable sequences in $l_{2}$ such that $\sum_{i=1}^{\infty}\alpha_{i}^{2}=1$.
Name as $B\subset l_{2}$ this set of $\alpha$'s, i.e. the unit sphere. 

Take $B_{k}$ in (\ref{SetBk}) and notice that for any point of $B_{k}$ one can construct an associated subspace $M_{1}$ of dimension $\leq k$, and any $M_{1}\in \mathcal{M}$ of dimension $\leq k$ can be associated to a point in $B_{k}$. 
However, two different points of $B_{k}$ can generate the same base $E$. 
Let us give the precise description of $\mathcal{M}$. 

First, let $B_{-k}$ be the set of $k$-tuples in $B_{k}$ that produce a base $E$ with $dim(E)<k$ (i.e. linearly dependant $k$-tuples), and define $V_{k}:=B_{k}- B_{-k}$. However, given a tuple $\hat{\alpha}\in B_{k}$ such that the generated base is of dimension $<k$, one can realize that any open set of $B_{k}$ containing $\hat{\alpha}$ has a tuple that generates a $k$-dimensional base. 
This implies that $V_{k}$ is dense in $B_{k}$.
Now, identify in equivalent classes from $V_{k}$ all the $k-$tuples that produce the same base $E$ with $dim(E)=k$ (i.e. permutations of a $k$-tuple), 
so that in the end $\mathcal{M}$ is a quotient space of $V_{k}$.

We proceed to identify what the space $B$ really is. 
It turns out that $B$ is homeomorphic to 
$$P:=(-1,1)\times (-1,1)\times\ldots$$ 
with the product-topology, see Proposition 10.1 of \cite{Bessaga}; 
in that reference the set $P$ is described in the Introduction and $B$ in page 10.1. 

It is known that the closure of $P$ is 
$$Q:=[-1,1]\times [-1,1]\times \ldots,$$ 
see Theorem 19.5 in \cite{Munkres}. 
Therefore, by the Tychonoff Theorem, $B$  is relatively compact, and therefore so is $B_{k}$. 
Finally, it is known that $Q$ is second-countable, which implies that it is separable. 
These properties of $B_{k}$ are inherited by the quotient space $V_{k}$, that is to say by $\mathcal{M}$.  
\end{proof}


\bigskip
Now, let us join previous pieces for the proof of our main theorem.
\bigskip

\begin{proof} (\textbf{of Theorem \ref{TheoremMain}}) From Lemma \ref{LemaFn} our aim is to see that 
\begin{equation*}
\lambda_{k}^{(n)}\to\lambda_{k}\text{ in distribution as } k\to\infty, 
\end{equation*}
for each $k=1,2,\ldots$, where the eigenvalues admit the representations (\ref{EqFn}) and (\ref{EqSupDisc}).

In Lemma \ref{LemaSup} we proved that $(\forall \alpha\in S^{k-1}) F_{n}(g_{n})\to F(g)$, where $g_{n}$ and $g$ are defined through plugging
$\alpha$ into the homomorphisms $\gamma_{n}$ and $\gamma$, respectively.  
Abusing of the notation let us shortly write $F_{n}(\alpha):=F_{n}(g_{n})$ and $F(\alpha):=F(g)$. 

From this convergence, for each fixed $\alpha$, we can extract almost surely convergent subsequences. 
Now, let us go a step further. 
Let $D$ be a dense numerable subset of $S^{k-1}$. 
Using the Cantor's diagonal procedure, we can extract a subsequence $n_{1},n_{2},\ldots$ such that almost surely
\begin{equation}\label{EqConFD}
\forall \alpha\in D,\ F_{n_{j}}(\alpha)\to F(\alpha),\ as\ j\to\infty.
\end{equation}
This is performed in the following way. 
Consider that $D=\{\alpha_{1}, \alpha_{2},\ldots\}$.
For $\alpha_{1}$ take the infinite set $I_{1}\subset \mathbb{N}$ such that $F_{n}(\alpha_{1})\to F(\alpha_{1})$ when $n\in I_{1}$ and $n\to\infty$. 
Now, call $n_{1}$ the first element of $I_{1}$. 
In a similar way, for $\alpha_{2}$ take an infinite set $I_{2}\subset I_{1}$ such that $F_{n}(\alpha_{2})\to F(\alpha_{2})$ when $n\in I_{2}$ and $n\to\infty$. 
Now, call $n_{2}$ the first element of $I_{2}$ with $n_{2}>n_{1}$. 
Continuing with this procedure one can construct a sequence of numbers $\{n_{j}\}_{j=1}^{\infty}$ where (\ref{EqConFD}) happens.


Now, since $D$ is dense and relatively compact, it should happen that almost surely
\begin{equation}\label{LimitFtilde}
\tilde{F}_{n_{j}}(M_{1}):=\sup_{\alpha\in S^{k-1}}F_{n_{j}}(\alpha)\to \tilde{F}(M_{1}):=\sup_{\alpha\in S^{k-1}}F(\alpha),
\ j\to\infty.
\end{equation}
The writing $\tilde{F}(M_{1})$ is to emphasize that this is done for $M_{1}$ fixed.

To see (\ref{LimitFtilde}), suppose that $\lim_{j}\tilde{F}_{n_{j}}(M_{1})\neq \tilde{F}(M_{1})$. 
Since $D$ is dense, we can take $\alpha_{n_{j}}^{*}\in D,\ j=1,2,\ldots$  such that $F_{n_{j}}(\alpha_{n_{j}}^{*})$ is as close as we wish to $\tilde{F}_{n_{j}}(M_{1})$, 
and we can also take $\alpha^{*}\in D$ such that we are as close as we wish to $\tilde{F}(M_{1})$. 
Hence, from the relative compactness of $D$, one can extract a subsequence $\{n_{r}\}_{r=1}^{\infty}$ from $\{n_{j}\}_{j=1}^{\infty}$ such that the sequence $\alpha_{n_{r}}^{*},\ r=1,2,\ldots$ converges. 
And from this one can derive a contradiction for the very definition of the supremum, either for 
$\sup F_{n_{r}}$ for some $r$ big enough or for $\sup F$. 

So far we have proved that for any $M_{1}\in \mathcal{M}$ fixed, there is a subsequence $\{n_{r},\ r=1,2,\ldots\}$ from the original one $n=1,2,\ldots$, such that
$\tilde{F}_{n_{r}}(M_{1})\to \tilde{F}(M_{1})$ almost surely, as $r\to\infty$. 
Going one step further, we now want to see if we can extract a new subsequence $\{n_{m}\}_{m=1}^{\infty}$ 
from $\{n_{r}\}_{r=1}^{\infty}$ such that
\begin{equation}\label{EqInf}
\inf_{M_{1}\in \mathcal{M}} \tilde{F}_{n_{m}}(M_{1}) \to \inf_{M_{1}\in \mathcal{M}}\tilde{F}(M_{1}), \ m\to\infty,
\end{equation}
almost surely.
And here is where we use the Lemma \ref{lemaM}, as it tells us that $\mathcal{M}$ can be seen as a relatively compact space with a countable dense subset. 
If we can check that the functions $\tilde{F}_{n_{m}}$ and $\tilde{F}$ are continuous, then we can use the same method for the supremum to extract the desired subsequence.
 
Let us argue how $\tilde{F}$ is continuous, because for $\tilde{F}_{n_{m}}$ is the same idea. 
Suppose that we are told that $M_{1}^{(n)}\to M_{1},\ n\to\infty$. 
When looking at the proof of Lemma \ref{lemaM}, we see that each $M_{1}^{(n)}$ is constructed using the same base $V$, and this construction is through elements in $l_{2}$, the coefficients of the linear combination. 
These coefficients converge in $l_{2}$ to those determining $M_{1}$.
This tells us that there are homeomorphisms $h_{n}:S\to S^{(n)}$, where $S^{(n)}$ is the unitary ball of $M_{1}^{(n)}$ and $S$ of $M_{1}$, such that $h_{n}(g)\to g$ for each $g\in S$. 
Therefore, using the same method for the supremum described above,
\begin{equation*}
\tilde{F}(M_{1}^{(n)})=\sup_{ g\in S^{(n)}}F(g)=\sup_{g\in S}F(h_{n}(g))\to \sup_{g\in S}F(g)= \tilde{F}(M_{1}),\ n\to\infty.
\end{equation*}
In the same manner we establish the continuity of $\tilde{F}_{n_{m}}$, which ultimately validates the limit (\ref{EqInf}).

The general conclusion is that given any sequence from $\{\lambda_{k}^{(n)},\ n=1,2,\ldots\}$, we can extract a subsequence that converges in distribution to the same law, namely the one of $\lambda_{k}$. 
Therefore, due to Theorem 2.6 of \cite{Billingsley}, the whole sequence converges in distribution.
\end{proof}

\bigskip

\textbf{Acknowledgements}. 
The author would like to thank Jes\'{u}s Gonz\'{a}lez, his former teacher in topology, for the help with some topological issues.



\begin{thebibliography}{9}


\bibitem{Bessaga}  C. Bessaga (1969). \textit{Topics from Infinite Dimensional Topology: $\mathcal{K}$-skeletons and Anderson's Z-sets}, 
Lecture Notes Series No. 18, Aarhus Universitext.	



\bibitem{Billingsley}  P. Billingsley (1999). \textit{Convergence of Probability Measures} (2nd ed.), Wiley.



\bibitem{Conway}  J.B. Conway (2007). \textit{A Course in Functional Analysis}, Springer.





\bibitem{Fuku}  M. Fukushima and S. Nakao (1977). On spectra of the Schr\"{o}dinder operator with a white Gaussian noise potential. 
\textit{Zeitshrift f\"{u}r Wahrsheinlichkeitstheorie} \textbf{37}, pp. 267--274.




\bibitem{Munkres}  J.R. Munkres (2000). \textit{Topology}, (2nd ed.), Prentice Hall.

\bibitem{Pacheco}  C.G. Pacheco (2016). A random matrix from a stochastic heat equation. \textit{Statistics and Probability Letters}. On line.



\bibitem{Ramirez}  J. Ram\'{i}rez, B. Rider and B. Vir\'{a}g (2011). Beta ensembles, stochastic Airy spectrum, and a diffusion. 
\textit{Journal of the American Mathematical Society} \textbf{24}, No. 4, pp. 919--644.


\bibitem{Schmudgen}  K. Schm\"{u}dgen (2012). \textit{Unbounded Self-adjoint Operators on Hilbert Space}, Springer.

\bibitem{Skorohod}  A.V. Skorohod (1984). \textit{Random Linear Operators}, D. Reidel Publishing Company.

\bibitem{Zhang}  F. Zhang (2011). \textit{Matrix Theory: Basic Results and Techniques} (2nd ed.), Springer.

\bibitem{Walsh}  J.B. Walsh (1986). An introduction to stochastic partial differential equations. In 
\textit{Ecole d'Et\'{e} de Prob. de St-Flour XIV}, Lectures Notes in Mathematics 1180, Springer, pp. 265--439.



\end{thebibliography}
\end{document}